\documentclass{article}
\usepackage{amssymb}
\usepackage{amsthm}
\usepackage{graphicx} 
\usepackage{hyperref}
\usepackage{mathtools}
\usepackage{fullpage}

\usepackage{minted} \BeforeBeginEnvironment{minted}{\begingroup\color{black}} \AfterEndEnvironment{minted}{\endgroup} \setminted{autogobble,breaklines,breakanywhere,linenos}
\usepackage{amsmath}
\usepackage{listings}

\DeclarePairedDelimiter{\norm}{\lVert}{\rVert}

\DeclarePairedDelimiter{\floor}{\lfloor}{\rfloor}

\title{Improved bounds for lines and $1$-separated sets in Euclidean Ramsey theory}

\author{Gabriel Currier\footnote{Department of Mathematics, University of British Columbia, 1984 Mathematics Road, Vancouver, BC, Canada V6T 1Z2. currierg@math.ubc.ca}, Param Mody\footnote{University of British Columbia, Vancouver, BC, Canada. pm13@student.ubc.ca}, Zehan Xie\footnote{University of British Columbia, Vancouver, BC, Canada. zehanxie@student.ubc.ca}, Jiaming Zhang\footnote{University of British Columbia, Vancouver, BC, Canada. jiaming.ubc@gmail.com}}
\date{September 2026}

\newtheorem{theorem}{Theorem}[section]
\newtheorem{lemma}[theorem]{Lemma}

\newcommand{\E}{\mathbb{E}}

\begin{document}

\maketitle

\begin{abstract}
    Let $K$ be a $1$-separated set of diameter at most $R-1$, and let $\ell_m$ denote a collection of $m$ points on a line, with consecutive points of distance $1$ apart. Conlon and Fox (2019) demonstrated a coloring of $n$-dimensional Euclidean space avoiding red congruent copies of $\ell_2$ and blue congruent copies of $K$ for $|K| > 10000^n\log_2 R$. We show here a stronger bound, that in fact $|K| > (6.79 + o(1))^n\log R$ suffices for arbitrary $1$-separated $K$, while the improvement $|K| > (5 + o(1))^n\log R$ holds in many cases, including when $K = \ell_m$, or more generally when $K$ is contained in a low-dimensional affine subspace. We also make a special study of the case when $n=2$, demonstrating a two-coloring of two-dimensional Euclidean space avoiding red copies of $\ell_2$ and blue copies of $\ell_{6330}$. This latter result addresses a question of Erd\H{o}s and Graham.
\end{abstract}
\section{Introduction}
We let $\mathbb{E}^n$ denote $n$-dimensional Euclidean space; that is, $\mathbb{R}^n$ equipped with the Euclidean norm. We say that a set is $t$-separated if all pairs of points are of distance at least $t$ apart, and let $\ell_m$ denote a collection of $m$ points on a line, with consecutive points of distance $1$ apart. We write that $\mathbb{E}^n \rightarrow (S_1,\dots,S_k)$ if any $k$-coloring of $\mathbb{E}^n$ contains a congruent copy of $S_i$ in color $i$, for some $1 \le i \le k$. If there is a coloring that avoids such configurations, we write that $\mathbb{E}^n \not \rightarrow (S_1,\dots,S_k)$. Finally, if $S = S_1 = \dots,S_k$, we write simply $\mathbb{E}^n \xrightarrow[]{k} S$ or $\mathbb{E}^n \not \xrightarrow[]{k} S$.

The question of for which $n,S_1,\dots,S_k$ we have $\mathbb{E}^n \to (S_1,\dots,S_k)$ has a long history. The most famous question in the area, known as the Hadwiger-Nelson problem (or sometimes the \emph{chromatic number of the plane}), asks for the largest integer $k$ such that $\mathbb{E}^2 \xrightarrow[]{k} \ell_2$. The best known bounds on this problem for many years were $3 \le k \le 6$, but a recent computational breakthrough has shown $k \ge 4$ \cite{dG18}.

After the introduction of the Hadwiger-Nelson problem by Nelson in 1950 (not in print), the field was further developed in a pioneering series of papers by Erd\H{o}s, Graham, Montgomery, Rothschild, Spencer and Straus \cite{E73,E75a,E75b}. Many problems were introduced in these papers, including a number about two-colorings of $\mathbb{E}^n$. This will be the focus of the present manuscript.

In particular, we will be interested in the following question: for which $1$-separated sets $K$ and for which positive integers $m$ do we have $\mathbb{E}^n \to (\ell_2,\ell_m)$ or $\mathbb{E}^n \to (\ell_2,K)$? Questions of this type were first considered in \cite{E73,E75a,E75b}, where they showed for example that $\mathbb{E}^2 \rightarrow (\ell_2,\ell_4)$ \cite{E75a} and that $\mathbb{E}^n \not \to (\ell_6,\ell_6)$ \cite{E73}. Since then, there has been a lot of work on these and related problems. Some notable advances are the following.

\begin{itemize}
    \item $\E^2 \to (\ell_2, K)$ for any $K$ with at most $4$ points (Juh\'{a}sz \cite{J79})
    \item There is a set $K$ with $8$ points such that $\E^2 \not \to (\ell_2, K)$ (Csizmadia and T\'{o}th \cite{CT94})
    \item $\E^2 \to (\ell_2, \ell_5)$ (Tsaturian \cite{T17})
    \item $\E^3 \to (\ell_2, \ell_6)$ (Arman and Tsaturian \cite{AT18})
    \item $\E^n \not \to (\ell_2, \ell_{c_1^n})$ (Conlon and Fox \cite{CF19}) and $\E^n \to (\ell_2, \ell_{c_2^n})$ (Szlam \cite{S01}) for constants $c_1 > c_2 > 1$.
    \item $\E^n \to (\ell_3,\ell_3)$ for $n \ge 2$ (Currier, Moore and Yip \cite{CMY}).
    \item $\E^n \not \to (\ell_3, \ell_{10^{50}})$ (Conlon and Wu \cite{CW23}), which was subsequently improved to $\E^n \not \to (\ell_3, \ell_{1177})$ (F\"uhrer and T\'oth \cite{FT24}) and $\E^n \not \to (\ell_3, \ell_{20})$ (Currier, Moore and Yip \cite{CurrierMooreYip2026Avoiding})
    \item $\E^n \not \to (\ell_4,\ell_{14})$ and $\E^n \not \to (\ell_5,\ell_{8})$ (Currier, Moore and Yip \cite{CurrierMooreYip2026Avoiding}) and $\E^n \not \to (\ell_6,\ell_6)$ (Erd{\H{o}}s et al. \cite[Theorem 12]{E73})
\end{itemize}

Conlon and Fox in \cite{CF19} actually proved something more general; they showed that any $1$-separated $K$ of diameter at most $R-1$ and size at least $10^{4n}\log_2 R$\footnote{Logarithms are assumed to be natural unless otherwise specified.} satisfies $\E^n \not \to (\ell_2,K)$. Taking $K = \ell_m$ with $m \ge 10^{5n}$ gives in particular that $\E^n \not \to (\ell_2,\ell_m)$. The goal of the present manuscript is to improve these bounds. Our first result is the following:

\begin{theorem}\label{thm:main_bign}
 There is an absolute constant $C >0$ such that the following holds. Suppose $R_K > 2$, and that $K$ is a $1$-separated set of diameter at most $R_K-1$. Suppose furthermore that any point in $K$ has at most $C_K$ other points of $K$ at distance $<5$, and that $|K| \ge C n^6\log R_K \max\{5^n,C_K\}$. Then $\E^n \not \to (\ell_2,K)$.
 \end{theorem}

The upcoming lemma \ref{lem:better_sep} shows that for any $1$-separated $K$ we have $C_K \lesssim 6.79^n,$\footnote{Generally, we say that $A \lesssim B $ if there exists an absolute constant $C$ such that $A \le CB$.} so as an immediate corollary we have $\E^n \not \to (\ell_2,K)$ as long as $|K| > (6.79+o(1))^n\log R_K$. Moreover, if $K$ is contained in a $d$-dimensional affine subspace, then lemma \ref{lem:better_sep} gives $C_K \lesssim 6.79^d$, and so if $d \le \frac{\log(5)}{\log(6.79)}n$ then the weaker condition $|K| > (5+o(1))^n\log R_K$ is sufficient. This in particular shows that $\mathbb{E}^n \not \to (\ell_2,\ell_m)$ for $m > (5+o(1))^n$. As noted in \cite{CF19}, an argument of Szlam from \cite{S01} combined with a result of Raigorodskii \cite{Raigorodskii2000} shows that $\E^n \to (\ell_2,K)$ for any $K$ of size at most $(1.239... + o(1))^n$. Thus, theorem \ref{thm:main_bign} substantially closes the gap between the upper and lower bounds, both for $\ell_m$ and for general $1$-separated $K$ of bounded diameter.

The case of $n=2$ and $K=\ell_m$ has received further attention in the literature; see for example the above results of Erd\H{o}s et al. \cite{E75a} and of Tsaturian \cite{T17}. Furthermore, Erd\H{o}s and Graham asked this question specifically in \cite{ErdosGraham1980OldNew}, where they note (without proof) that $\E^2 \not \to (\ell_2,\ell_m)$ for some $m$ ``at most 10000000.'' The result of Conlon and Fox in \cite{CF19} implies that $\E^2 \not \to (\ell_2,\ell_{10^{10}})$, although a better bound can be read out of their proof without significantly changing the argument. We show the following further improved bound.

\begin{theorem}\label{thm:main_smalln}
    $\E^2 \not \to (\ell_2,\ell_{6330})$.
\end{theorem}

Our colorings follow the same general outline as those from \cite{CF19}. However, we are able to base our colorings on lattices instead of the greedy point sets used in \cite{CF19}, as well as use more sophisticated probabilistic and geometric arguments to substantially improve the bounds. This is especially true in the case of theorem \ref{thm:main_smalln} where we are also able to exploit the simpler structure of $\mathbb{E}^2$ and $\ell_m.$

In Section \ref{sec:prelim}, we will introduce various preliminaries from probability, geometry and algebra that we will need for our arguments. In section \ref{sec:mains}, we'll describe the common framework used for proving theorems \ref{thm:main_bign} and \ref{thm:main_smalln}, and then in sections \ref{subsec:low_d} and \ref{subsec:bigd} we'll execute these two proofs. For the remaining sections, if we refer to a \emph{copy} of a given configuration, we are referring to a congruent copy.

\section{Preliminaries and Notation}\label{sec:prelim}

In the upcoming sections, we will need a few tools from geometry and probability to complete the proof. For the remainder of the paper, $Vol_i(V)$ refers to the $i$-dimensional volume of the set $V$. If $i = n$ we will write just $Vol(V)$. Furthermore, we will use $B_r^n(q)$ to denote the (closed) $n$-dimensional ball of radius $r$ centered at a point $q \in \mathbb{E}^n$. If we write simply $B_r^n$, this refers to the ball centered at $0$. We will need the following lemmas from geometry. To start, we will need to prove a packing-type bound for $1$-separated sets.

\subsection{A packing lemma for $1$-separated sets}\label{sec:packing}

The main result of this section tells us that subsets of $t$-separated sets cannot be packed too closely. A crude but useful version of this principle is the following.

\begin{lemma}\label{lem:s_t_sep}
Let $K \subset \mathbb{E}^n$ be $t$-separated. Then, for any point $q \in \mathbb{E}^n$ and $r \ge 0$, there are at most $(2r/t+1)^n$ points of $K$ within distance $r$ of $q$.
    
\end{lemma}

\begin{proof}
    Since $K$ is $t$-separated, a collection of (open) balls of radius $t/2$ placed around points of $K$ will all be disjoint. Furthermore, for any point within distance $r$ of $q$, these balls will be contained in $B_{r+t/2}^n(q)$. A volume argument shows that the number of such points must be at most $$\frac{Vol(B_{r+t/2}^n(q))}{Vol(B_{t/2}^n)} = \left(\frac{r+t/2}{t/2}\right)^n = \left(\frac{2r}{t}+1\right)^n$$
\end{proof}

We will wish to give a stronger bound than lemma \ref{lem:s_t_sep}. To do so, we will need a couple of definitions. For any $0 < \theta< \pi/2$ we let
\[
F(\theta)
:=
\frac{1+\sin\theta}{2\sin\theta}
\log\left(\frac{1+\sin\theta}{2\sin\theta}\right)
-
\frac{1-\sin\theta}{2\sin\theta}
\log\left(\frac{1-\sin\theta}{2\sin\theta}\right)
\]
and for each $0 \le \eta < 1/2$ and $r > 1$, we define
\[
\theta_{\eta,r}
:=
2\arcsin\left(
    \frac{\sqrt{1-\eta^2}}{2r}
\right).
\]
Our goal is the following. For convenience, we state this only for $1$-separated sets, but a corresponding result can be proven for $t$-separated sets.

\begin{lemma}\label{lem:better_sep}

    Let $K \subset \mathbb{E}^n$ be $1$-separated and fix $r > 1$. Then, for any point $q \in \mathbb{E}^n$, the number of points of $K$ within distance $r$ of $q$ is at most $(\exp F(\theta_{0,r}) + o(1))^n$.
\end{lemma}

For our primary application we will apply lemma \ref{lem:better_sep} with $r=5$, so we note that $\exp F(\theta_{0,5}) = 6.7844...$ for later use. To begin with the proof, we will define $A(n,\theta)$ to be the maximum size of a set in $\mathbb{R}^n\setminus \{0\}$ where every pair of points makes an angle of at least $\theta$ with the origin. Equivalently, we can write $$A(n,\theta) = \max\{|U| : U \subset S^{n-1} \text{ and } \langle u,v\rangle \le \cos \theta  \text{ for all distinct } u,v \in U\}.$$ For an upper bound on $A(n,\theta)$, we will use theorem $4$ from \cite{KabatyanskiiLevenshtein1978}; for an exposition of this result, see equation $2.6$ of \cite{CohnZhao2014}. This says that for any fixed  \(0<\theta<\pi/2\) we have
\begin{align}
\frac{1}{n}\log A(n,\theta) \label{line:Cohn_Zhao}
\leq
F(\theta) + o(1).
\end{align}
We will use this to prove lemma \ref{lem:better_sep}. To do so, we define $M_r^n$ to be the maximum size of a $1$-separated set contained in $B_r^n$. We will prove the following bound on $M_r^n$.

\begin{lemma}\label{lem:shell}

Let $0 < \eta < 1/2$ and $r > 1$. Then, 
\(
M_r^n
\leq
\left(\floor*{\frac{r}{\eta}}+1\right)
A(n,\theta_{\eta,r}).
\)
\end{lemma}

\begin{proof}
 Let $K$ be a $1$-separated set contained in $B_r^n$. For any $j \in \mathbb{Z}^{\ge 0}$, we denote $$K_j := \{q \in K: \norm{q} \in [j\eta,(j+1)\eta)\}$$ Our goal is to show that $|K_j| \le A(n,\theta_{\eta,r})$ for each $j$. Note that $K_0$ has diameter at most $2 \eta < 1$, and since $K$ is $1$-separated it follows that $|K_0| \le 1 \le A(n,\theta_{\eta,r}).$ Thus, we may assume $j \ge 1$ and thus all points are non-zero. 

Now, we want to show that the points of $K_j$ have angular separation at least $\theta_{\eta,r}$, and the claim will follow. Let distinct $x,y \in K_j$ have angular separation $\theta$. We can derive from the law of cosines that $$\sin^2(\theta/2) = \frac{||x-y||^2-(||x||-||y||)^2}{4||x||||y||} \ge \frac{1-\eta^2}{4r^2}$$ and thus $\sin(\theta/2) \ge \frac{\sqrt{1-\eta^2}}{2r}.$ Since $\theta/2 \in [0,\pi/2]$ and $\sin$ is increasing on this interval, we conclude that $\theta \ge 2 \arcsin\left(\frac{\sqrt{1-\eta^2}}{2r}\right)$. Thus, $|K_j| \le A(n,\theta_{\eta,r})$.

Finally, we observe that since $K \subset B_r^n$, we have that $|K_j| = 0$ for any $j > \lfloor r/\eta\rfloor$. Thus, we conclude $$|K| = \sum_{j = 0}^{\lfloor r/\eta\rfloor} |K_j| \le \left(\lfloor r/\eta\rfloor+1\right) A(n,\theta_{\eta,r})$$

\end{proof}

    Now, we prove lemma \ref{lem:better_sep}

    \begin{proof}[Proof of lemma \ref{lem:better_sep}]

        Let $K$ and $q$ be as in the statement of lemma \ref{lem:better_sep}, and let $K_{q,r} := \{q' \in K : ||q-q'|| < r\}$. Translating $q$ to the origin, we see clearly that $|K_{q,r}| \le M_r^n$. Using Lemma \ref{lem:shell} we see $$(1/n)\log M_r^n \le (1/n)\log\left(\lfloor r/\eta\rfloor+1\right) + (1/n) \log A(n,\theta_{\eta,r}).$$ We observe $(1/n)\log\left(\lfloor r/\eta\rfloor+1\right) \to 0$ as $n \to \infty$. Using (\ref{line:Cohn_Zhao}), we see that for a fixed $0 <\eta < 1/2$, we get $$(1/n)\log M_r^n \le F(\theta_{\eta,r}) + o(1).$$ Since $F$ is continuous and $\theta_{\eta,r} \to \theta_{0,r}$ as $\eta \to 0$, we see $$(1/n)\log M_r^n \le F(\theta_{0,r}) + o(1).$$ Using $|K_{q,r}| \le M_r^n$ completes the claim.
     \end{proof}

\subsection{Other auxiliary lemmas}

First, we will need an inequality due to Suen (and refined by Janson) \cite{Janson1998Suen,Suen1990Correlation} from probability. The theorem deals with the case when we have a sequence of events which are mostly (but not entirely) independent from one another, and we want to bound the probability that none of them happens. It says that, as long as the dependencies are relatively infrequent, then the probability that none of them happens is close to what it would be if they were independent, up to a small error term. 

More formally, suppose we have indicator variables $X_1,\dots,X_m,$ where $\mathbb{E}[X_i] = p_i$ and $X := \sum X_i$. We say that $i \sim j$ is a dependency graph on $[m]$ if, given families $I,J \subset [m]$ with no edges between $I,J$, the families $\{X_i\}_{i \in I}$ and $\{X_j\}_{j \in J}$ are independent. Then, we define $\mu = \mathbb{E} [X] = \sum p_i$ and $\Delta = \sum_{i \sim j} \mathbb{E}[X_iX_j]$ (where the sum is taken over unordered pairs $i,j$) and $\delta_i = \sum_{j \sim i} p_j$ and $\delta = \max_i \delta_i$. In \cite{Janson1998Suen}, Janson proved the following bound on $\mathbb{P}(X = 0)$.

\begin{theorem}\label{thm:Suen}
$\mathbb{P}
(X=0) \le e^{-\mu + \Delta e^{2\delta}}.$
\end{theorem}

We will also need the following version of the Milnor-Thom theorem (see \cite{Milnor1964BettiRealVarieties,OleinikPetrovskii1949TopologyRealAlgebraicSurfaces,Thom1965HomologieVarietesAlgebriquesReelles}, as well as \cite{Matousek2002LecturesDiscreteGeometry} for an overview), which will help us determine the number of connected components in $\mathbb{E}^n$ determined by a collection of surfaces. This is often stated in the language of sign patterns; for a collection of polynomials $F_1,\dots,F_M, $ a \emph{sign pattern} is a vector $v \in \{-1,0,1\}^M$ where there exists $q \in \mathbb{E}^n$ such that $F_i(q)$ matches the sign of $v(i)$ for each $1 \le i \le M$. 
\begin{theorem}\label{thm:m_t}For \( M \geq N \geq 2 \), the number of sign patterns of \( M \) polynomials in \( N \) variables, each of degree at most \( D \), is at most 
$\left( \frac{50DM}{N} \right)^{N}.$
\end{theorem}

For a centrally-symmetric convex body $V\subset \mathbb{E}^d$, we define the inradius $r(V)$ to be the supremum of the radii of all inscribed spheres, and the outradius $R(V)$ to be the infimum of the radii of all circumscribed spheres. The following lemma allows us to bound the volume of such a body in terms of its inradius and outradius.

\begin{lemma}\label{lem:conv_vol}
    Let $V \subset \mathbb{E}^n$ be a centrally symmetric convex body. Then, $$Vol(V) \le 2r(V)Vol_{n-1}(B_{R(V)}^{n-1})$$
    
\end{lemma}

\begin{proof}
We start by translating $V$ so its center is at the origin, and let $h_V(u)=\max_{x\in V}\langle x,u\rangle$ denote the support function of $V$. Since $V$ is closed, convex, and centrally symmetric we can write it as $$V = \bigcap_{u \in S^{n-1}} \{x \in \mathbb{E}^n : |\langle x,u\rangle| \le h_V(u)\}.$$ Then, there must exist $u_0 \in S^{n-1}$ such that $h_V(u_0) = r(V)$, so by extension $$V \subseteq  \{x \in \mathbb{E}^n : |\langle x,u_0\rangle| \le r(V)\}.$$ Let $\pi:\mathbb R^n\to u_0^\perp$ be the orthogonal projection onto $u_0^\perp$. Every fiber of $\pi|_V$ is a line segment of length at most $2r(V)$, and therefore
\begin{align}\label{line:Vol1}
    Vol(V)\le 2r(V)\,Vol_{n-1}(\pi(V)).
\end{align}
Finally, since $V\subseteq B^n_{R(V)}$, we have
\begin{align}\label{line:Vol2}
\pi(V)\subseteq \pi(B^n_{R(V)})=B^{n-1}_{R(V)},
\end{align}
and combining (\ref{line:Vol1}) and (\ref{line:Vol2}) completes the claim.
\end{proof}

In the following sections, we will divide $\mathbb{E}^n$ into (connected) cells such that each cell has diameter $<1$. For a given cell $D$, we will let $z(D)$ denote the collection of cells $D' \neq D$ such that the minimum distance over all $q \in D, q'\in D'$ is at most $1$.

\begin{lemma}\label{lem:5sep}
    Let $q,q' \in \mathbb{E}^n$ be contained in cells $D,D'$ (respectively), and suppose $||q-q'|| \ge 5$. Then $(\{D\} \cup z(D)) \cap (\{D'\} \cup z(D')\} = \emptyset$.
\end{lemma}

\begin{proof}
Suppose for the sake of contradiction that there exists a cell $D^* \in z(D) \cap z(D')$. This implies there exist points $q_1 \in D, q_2,q_3 \in D^*$ and $q_4 \in D'$ such that $||q_1-q_2|| \le 1$ and $||q_3-q_4|| \le 1$. Furthermore, since each cell has diameter $<1$, we also have $||q-q_1|| <1$ and $||q_2-q_3|| < 1$ and $||q_4-q'|| < 1$. By the triangle inequality
$$||q-q'|| \leq ||q-q_1|| + ||q_1-q_2|| + ||q_2-q_3|| + ||q_3-q_4|| + ||q_4-q'|| < 5,$$ which is a contradiction.
\end{proof}

Finally, we will do a brief introduction to lattices. Given linearly independent vectors $v_1,\dots,v_d$ in $\mathbb{E}^n$, we refer to the set $$\Lambda := \{\sum_{i}a_iv_i : a_i \in \mathbb{Z}\}$$ as a \emph{lattice}. The integer $d$ is known as the \emph{rank} of the lattice, and if $d=n$ we say the lattice is \emph{full rank}. Henceforth we will deal only with full rank lattices. Now, any set of vectors $b_1,\dots,b_d$ where $$\Lambda = \{\sum_{i}a_ib_i : a_i \in \mathbb{Z}\}$$ is said to be a \emph{basis} for $\Lambda$. Given such a basis, the set $$P = \{\sum_{i}a_ib_i : 0 \le a_i < 1\}$$ is known as a \emph{fundamental parallelepiped} of $\Lambda$. Furthermore, the \emph{Voronoi cell} $V_q$ of a point $q \in \Lambda$ is defined as $$V_q :=\{q' \in \mathbb{E}^n : ||q-q'|| < ||q^*-q'|| \text{ for all } q^* \in \Lambda \setminus \{q\}\}.$$ Note that for any $q\in \Lambda$, $V_q$ is a translation of $V_0$. It is a fundamental fact of lattices that for any $q \in \Lambda$ and basis $b_1,\dots,b_n$, we have $Vol(V_q) = |\det(\Lambda)|$, where $\det(\Lambda)$ denotes the determinant of the matrix with columns $b_1,\dots,b_n$. Finally, the \emph{$i$th successive minimum} $\lambda_i$ of $\Lambda$ is defined to be the smallest positive real number such that $B_{\lambda_i}^n \cap \Lambda$ spans an $i$-dimensional subset of $\mathbb{E}^n$.

For the results of section \ref{subsec:bigd}, we will need to find a basis for a lattice that has relatively small basis vectors. The Lenstra-Lenstra-Lov\'asz algorithm, for example, gives a basis where the length of the basis vectors is bounded by the successive minima $\lambda_i$, see \cite[page $48$]{NguyenVallee2010LLL}. 

\begin{theorem}\label{thm:LLL}
    Given a (full rank) lattice $\Lambda \subset \mathbb{E}^n$, there exists a basis $b_1,\dots,b_n$ of $\Lambda$ where $||b_i|| \le 2^{(n-1)/2}\lambda_i$ for each $1 \le i \le n$.
\end{theorem}

Theorem \ref{thm:LLL} is almost certainly not optimal for our purposes, but this loss will not have a significant effect on the final bound. In the next section we will proceed with the proofs of theorems \ref{thm:main_bign} and \ref{thm:main_smalln}.

\section{Proofs of main theorems}\label{sec:mains}

The colorings provided for theorems \ref{thm:main_bign} and \ref{thm:main_smalln} share the following common structure. For each, we will start with a lattice $\Lambda$ in $\mathbb{E}^n$, and scale $\Lambda$ so its Voronoi cells have diameter $<1$. Then, we will fix a large integer $L$, and color $\mathbb{T}_{L\Lambda} := (\mathbb{E}^n/L\Lambda)$ in a way that avoids red copies of $\ell_2$, and blue copies of $K$. This extends to a coloring of $\mathbb{E}^n$ with the same properties. 

Note that since $L$ is an integer, we guarantee that the Voronoi cells defined by the lattice $\Lambda$ in $\mathbb{T}_{L\Lambda}$ are the same as they would be in $\mathbb{E}^n$. Denote by $\mathcal{D}_\Lambda$ this set of cells in $\mathbb{T}_{L\Lambda}$. We will end up coloring each point in $\mathbb{T}_{L\Lambda}$ according to which element of $\mathcal{D}_\Lambda$ it is contained in. For this reason, we would also like to assign points on the boundary to a given cell. For every $i < n$, we will take each $i$-face\footnote{For a discussion of the $i$-faces of a hyperplane arrangement, see \cite[Section $6.1$]{Matousek2002LecturesDiscreteGeometry}.}, and assign the whole face to one of its neighboring cells. This is important for technical reasons in the proofs of lemmas \ref{lem:lowd_ad} and \ref{lem:highd_ad}.

 Our coloring will be chosen as follows: we fix a probability $p$, and choose a subset $Q$ of $\mathcal{D}_\Lambda$ by choosing each element of $\mathcal{D}_\Lambda$ independently with probability $p$. Finally, we choose a subset $S$ of $Q$, where each cell $D \in Q$ is included in $S$ if none of the cells in $z(D)$ is in $Q$. The points in cells in $S$ are those that we will color red. We note a few properties of these colorings.

\begin{enumerate}
    \item In any such coloring, there will be no red copy of $\ell_2$. This follows directly from the fact that the Voronoi cells have diameter $< 1$, and the definition of $S$.
    \item Let $u_\Lambda$ be the shortest vector in $\Lambda$; we will choose $L$ such that $L||u_\Lambda|| > R_K+4$. This guarantees first that the points in any copy of $K$ will all be in different cells. To see this, observe that the only way that two points can be in the same cell in $\mathbb{T}_{L\Lambda}$ is if they are distance $< 1$ apart, or distance at least $L||u_\Lambda||-1$ apart. The first cannot happen since $K$ is $1$-separated, and the second cannot happen by our choice of $L$, since $K$ has diameter $R_K$. In a similar manner, this condition will guarantee that the number of points at distance $< 5$ from any point of a copy of $K$ in $\mathbb{T}_{L\Lambda}$ is the same as it would be in $\E^n$. We will use this in lemmas \ref{lem:R2_allblue} and \ref{lem:high_d_prob}.
    \item For each $D \in \mathcal{D}_\Lambda$, the probability that $D \in S$ (and thus that the points of $D$ are colored red) is precisely $p(1-p)^{|z(D)|}$. Moreover, since $\Lambda$ is a lattice, $|z(D)|$ is the same for each cell $D$.
\end{enumerate}

The remainder of the proofs will be dedicated to showing that there exists such a coloring with no blue copy of $K$. This will proceed in two parts: first we will show that there are not too many collections of cells in $\mathcal{D}_\Lambda$ that can contain a copy of $K$, and second we will show that the probability that such a collection is colored all-blue is small. We will do this first for the case $K = \ell_m$ and $n=2$ in section \ref{subsec:low_d}, and then for the general case in \ref{subsec:bigd}.

\subsection{Case $n=2$ and $K =\ell_m$}\label{subsec:low_d}

In this section we prove theorem \ref{thm:main_smalln}. To start, we choose a small $\epsilon > 0$, and let $\Lambda$ denote the lattice in $\mathbb{E}^2$ spanned by the vectors $b_1:=(1-\epsilon)v_1$ and $b_2 :=(1-\epsilon)v_2$, where 
\[
v_1 := (3/2,0),\,v_2:=(3/4,\sqrt{3}/4).
\]
Note that this is a hexagonal lattice; that is, the Voronoi cells of $\Lambda$ are regular hexagons tiling the plane. Our choice of scaling ensures that each hexagon has diameter $(1-\epsilon) < 1$. Set $L = 3m$. As described in the previous section, we need that $L||u_\Lambda|| > m+4$. For this choice of $\Lambda$ we have that $u_\Lambda$ has length $\ge 1/2$ (assuming $\epsilon$ is chosen to be small enough), so this condition is indeed satisfied.

 As described in the previous section, we will now show that there are not too many collections of cells that can contain a copy of $\ell_m$. Let $q_1,\dots,q_m$ denote the points of $\ell_m$. We say that a collection of cells $D_1,\dots,D_m \subset \mathcal{D}_\Lambda$ is \emph{admissible} if there is a copy of $\ell_m$ in $\mathbb{T}_{L\Lambda}$ where $q_i \in D_i$.

\begin{lemma}\label{lem:lowd_ad}
    The number of admissible collections of cells is at most $2^45^{16}m^8$
\end{lemma}

\begin{figure}
    \centering \includegraphics[scale=.5]{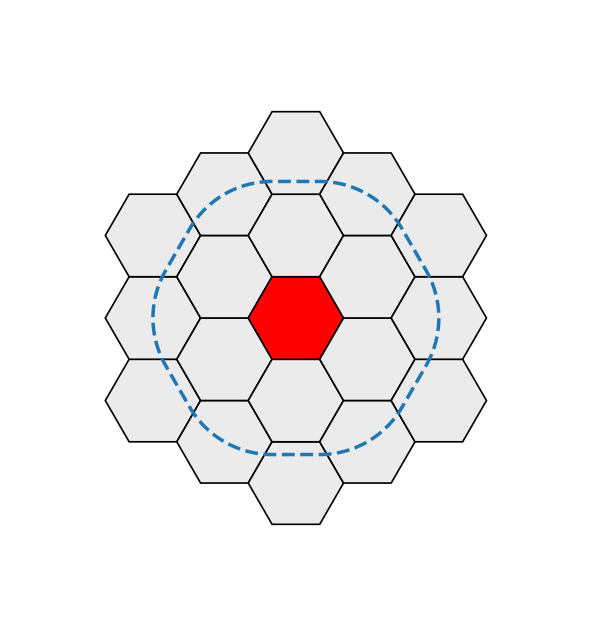}
    \caption{A cell $D$ colored red. The cells touching the interior of the dotted region are those that are included in $z(D)$. We can see that $|z(D)| = 18$.}
    \label{fig:hex}
\end{figure}

\begin{proof}

Let $P_0$ denote a translation of the fundamental parallelepiped of $L\Lambda$ (with basis \{$Lb_1$, $Lb_2$\}) so that its center is at the origin; that is $$P_0 = \{\sum_i a_iLb_i : -1/2 \le a_i < 1/2\}.$$ Suppose we have an admissible collection of cells $D_1,\dots,D_m$ with a corresponding copy of $\ell_m$ such that $q_i \in D_i$. We will construct a set of Voronoi cells $D_1',\dots,D_m'$ of $\Lambda$ in $\mathbb{E}^2$ as follows: Let $\ell_m' = \{q_1',\dots,q_m'\}$ be a copy of $\ell_m$ in $\mathbb{E}^2$ where each $q_i'$ is in the equivalence class of $q_i$, and at least one point from $\ell_m'$ is in $P_0$. Let $D_i'$ be the cell in $\mathbb{E}^2$ which contains $q_i'
$. We make the following observations.

\begin{enumerate}
    \item Every point in $P_0$ is within distance $7m/2$ of the origin, and since $\ell_m'$ has diameter $m-1$, every point in $\ell_m'$ is within distance $9m/2$ of the origin, so every point of every $D_i'$ is within distance $5m$ of the origin.
    \item The hexagonal Voronoi cells of $\Lambda$ have walls bounded by lines in $3$ directions, where the lines in each direction are spaced $\ge 1/3$ apart (assuming $\epsilon$ is chosen to be sufficiently small). Thus, the total number of lines bounding Voronoi cells within $5m$ of the origin is at most $2 \times 3 \times 3 \times5m \le 100m$. We denote these lines $F_1,\dots,F_s$.
    \item We let $q_1' = (x_1,y_1)$ and $q_2' = (x_2,y_2)$, and note that each other $q_i'$ can be expressed as a linear combination of $q_1'$ and $q_2'$. Thus, $F_i(q_j')$ for $1 \le i \le s $ and $1 \le j \le m$ corresponds to a sign pattern of at most $100m^2$ hyperplanes in variables $x_1,x_2,y_1,y_2$. Moreover, from this sign pattern, one can uniquely recover $D_1',\dots,D_m'$ and thus $D_1,\dots,D_m$. Here, we have used the fact from section \ref{sec:mains} that we assigned each lower-dimensional face to exactly one of its neighboring cells.

\end{enumerate}

Thus, the number of sign patterns of this arrangement will serve as an upper bound on the number of admissible collections of cells. By theorem \ref{thm:m_t}, the number of such sign patterns is at most $2^45^{16}m^8$, thus completing the proof.

\end{proof}
For the next lemma, we note that for any cell $D \in \mathcal{D}_\Lambda$, we have $|z(D)| = 18$. See Figure \ref{fig:hex} for an illustration of this. Furthermore, for the remainder of the proof we set $p=\frac{1}{19}$.

\begin{lemma}\label{lem:R2_allblue}
    The probability that an admissible collection of cells is all colored blue is at most $$ \exp (-0.01557m)$$
\end{lemma}

\begin{proof}
    The proof will proceed by lemma \ref{thm:Suen}. We let $D_1,\dots,D_m$ be an admissible collection of cells with points $q_1,\dots,q_m$ forming a copy of $\ell_m$ with $q_i \in D_i$. Furthermore, we let $X_i$ be the indicator variable for the event that $D_i$ is colored red, with $X = \sum_i X_i$. Thus, $X=0$ corresponds to the event that all cells are colored blue. We note that $\mathbb{E}[X_i] = p(1-p)^{18}$ for each $i$, and so $\mu = \E[X] = mp(1-p)^{18}$. Now, for any $1 \le i,j \le m,$ we say $i \sim j$ if $|i-j| < 5$. We note that this is in fact a dependency graph by the definition given in the setup of theorem \ref{thm:Suen}. To see this, let $I,J \subset [m]$ where $|i-j| \ge 5$ for all $i \in I$ and $j \in J$. We note that by lemma \ref{lem:5sep}, $$\left(\bigcup_{i \in I} \left(\{D_i\} \cup z\left(D_i\right)\right)\right) \bigcap \left(\bigcup_{j \in J} \left(\{D_j\} \cup z\left(D_j\right)\right)\right) = \emptyset.$$ Since $\{X_i\}_{i \in I}$ depends only on the cells in  $\cup_{i\in I} (\{D_i\} \cup z(D_i))$, and $\{X_j\}_{j \in J}$ depends only on the cells in $\cup_{j \in J} (\{D_j\} \cup z(D_j))$, we conclude that $\{X_i\}_{i \in I}$ and $\{X_j\}_{j \in J}$ are independent.

    It remains to calculate $\Delta$ and $\delta$. We note that, for each $ 1\le i \le m-1$, since $||q_i - q_{i+1}|| = 1$, we have $\E[X_iX_{i+1}] = 0$. Furthermore, by Lemma \ref{lem:5sep}, if $|i-j| \ge 5$, then $i \not \sim j$. Finally, we note if $2 \le |i-j| \le 4$, then $\E[X_iX_{j}] \le p^2(1-p)^{18}$. Therefore, we calculate

\begin{align*}
    \Delta = \sum_{i \sim j} \mathbb{E}[X_iX_j] \le (1/2)m(6p^2(1-p)^{18}) \le 3mp^2(1-p)^{18}
\end{align*}
and furthermore, for a fixed $i$ we have 

\begin{align*}
    \delta_i = \sum_{i \sim j}p(1-p)^{18} \le 8p(1-p)^{18}.
\end{align*}
Thus $\delta \le 8p(1-p)^{18}$ as well. Now, by theorem \ref{thm:Suen}, we get

$$\mathbb{P}(X=0) \le \exp(-\mu+\Delta e^{2\delta})$$
and a quick calculation with $p=1/19$ tells us that $-\mu+\Delta e^{2\delta} \le -0.01557m$, which completes the proof.

\end{proof}

To complete the proof of theorem \ref{thm:main_smalln}, we note if a coloring has no admissible collection of cells colored all blue, then it contains no all-blue $\ell_m$. Combining lemmas \ref{lem:lowd_ad} and \ref{lem:R2_allblue}, the expected number of admissible collections of cells that are colored all blue is bounded above by $$\exp (-0.01557m)2^45^{16}m^8.$$ A quick calculation shows that for $m \ge 6330$, the above value is $<1$. Therefore, there exists a two-coloring of $\mathbb{E}^2$ avoiding red copies of $\ell_2$ and blue copies of $\ell_{6330}$.

\subsection{General Case}\label{subsec:bigd}

In this section we will prove theorem \ref{thm:main_bign}. Let $K \subset \mathbb{E}^n$ be a $1$-separated set, and let $C_K$ be an integer such that the number of points in $K$ of distance $<5$ from any $q \in K$ is at most $C_K$. We will want a lattice $\Lambda$ with good covering properties. The \emph{covering radius} $r_\Lambda$ of $\Lambda$ is defined to be the smallest positive real number such that any point of $\mathbb{E}^n$ is within distance $r_\Lambda$ of some point of $\Lambda$. We are interested in lattices where the Voronoi cells have similar volume to the $n$-dimensional ball with radius $r_\Lambda$. For the state of the art, see a very recent paper of Li and Liu \cite{LiLiu2026} (and before that \cite{GaoLiuPikhurkoSun} of Gao, Liu, Pikhurko and Sun), improving on a classical result of Rogers \cite{Rogers1959LatticeCoverings}. In this paper, they construct lattices $\Lambda$ with covering radius $r_\Lambda$, where 

\begin{align*}
\frac{Vol(B_{r_\Lambda}^n)}{|\det(\Lambda)|} \lesssim n(\log n)^{1+o(1)}.
\end{align*}
We will not need such a strong bound, and will instead use the cruder estimate $\lesssim n^2$ for simplicity in our upcoming calculations. We take such a lattice $\Lambda$, scaled so $r_\Lambda= (1-\epsilon_n)/2$ for $\epsilon_n > 0$ small enough that we still have 
\begin{align}\frac{Vol(B_{1/2}^n)}{n^2} \lesssim |\det(\Lambda)|.\label{line:cel_vol}
\end{align}
Note that this also implies the diameter of each Voronoi cell is $< 1$. Furthermore, it is not too hard to see that for each successive minimum $\lambda_i$ of $\Lambda$, we have that $\lambda_i \le 2r_\Lambda\le 1$. Thus, theorem \ref{thm:LLL} guarantees us a basis $B = \{b_1,\dots,b_n\}$ where  $||b_i|| \le 2^{n/2}$ for each $i$. In the upcoming two lemmas, we prove some basic facts about our lattice.

\begin{lemma}\label{lem:short_vec}The shortest vector in $\Lambda,$ $u_\Lambda$ has length $\gtrsim\frac{1}{n^3}$. Equivalently, $\Lambda$ is $\gtrsim\frac{1}{n^3}$-separated.
\end{lemma}

\begin{proof}
    Let $V_0$ denote the Voronoi cell of $0$. The Voronoi cell is a centrally-symmetric convex body, where $r(V_0) = ||u_\Lambda||/2$. Furthermore, $R(V_0) \le 1/2,$ since the diameter of each Voronoi cell is $< 1$. Thus, we may use lemma \ref{lem:conv_vol} and (\ref{line:cel_vol}) to get $$\frac{Vol(B_{1/2}^n)}{n^2} \lesssim|\det(\Lambda)| = Vol(V_0) \lesssim ||u_\Lambda||Vol(B_{1/2}^{n-1}),$$ and we conclude $$||u_\Lambda|| \gtrsim\frac{1}{n^3}. $$
\end{proof}

\begin{lemma}\label{lem:z_size_K}
    For any Voronoi cell $D$ in $\mathcal{D}_\Lambda$, we have that $|z(D)| \lesssim 5^n n^2$
\end{lemma}

\begin{proof}
    Let $q$ be the center of $D$. We claim that any cell $D' \in z(D)$ is contained in $B_{5/2}^n(q)$. Note that by definition of $z(D)$, there exists a point $q
    _1\in D$ and another point $q_2 \in D'$ such that $||q_1-q_2|| \le 1$. Take an arbitrary point $q_3 \in D'$. Since $D,D'$ have diameter less than $1$, we know that $||q_2 - q_3|| < 1$ and $||q-q_1|| < 1/2$. Thus by the triangle inequality $$||q_3-q|| \le ||q - q_1|| + ||q_1 - q_2|| + ||q_2 - q_3|| < 1/2 + 1 + 1 = 5/2$$
     which shows that $D' \subset B_{5/2}^n(q)$. All Voronoi cells have the same volume, so by a volume argument and using (\ref{line:cel_vol}) we conclude $$|z(D)| \le \frac{Vol(B_{5/2}^n(q))}{Vol(D)} \lesssim \frac{Vol(B_{5/2}^n(q))n^2}{Vol(B_{1/2}^n)} \lesssim 5^{n} n^2.$$
\end{proof}

As before, we need to choose an integer $L$ such that $L||u_\Lambda|| \ge R_K + 4$; in particular, by Lemma \ref{lem:short_vec} choosing some $L \lesssim R_Kn^3$ will suffice. Now, we can move on to the final part of the proof. In particular, we do the same as in section \ref{subsec:low_d}; that is, give an upper bound on the number of admissible collections of cells, and give an upper bound on the probability that an admissible collection is colored all-blue. As a reminder, a collection of Voronoi cells $D_1,\dots,D_{|K|}$ in $\mathcal{D}_\Lambda$ is called admissible if there exists a copy of $K = \{q_1,\dots,q_{|K|}\}$ with $q_i \in D_i$ for all $1 \le i \le |K|$. In this lemma we will aim for simplicity instead of optimal bounds, since this will not have a significant effect on the final result.

\begin{lemma}\label{lem:highd_ad}
    The number of admissible collections of cells is $\lesssim R_K^{2n^3}4^{3n^4}|K|^{2n^2}$ 
\end{lemma}

\begin{proof}
Let $P_0$ denote the fundamental parallelepiped of $L\Lambda$ with basis $LB$, translated so that its center is at the origin; that is $$P_0 = \left\{\sum Lt_ib_i : -1/2 \le t_i < 1/2\right\}.$$ Let $D_1,\dots,D_{|K|}$ be an admissible collection of cells, with a copy of $K=\{q_1,\dots,q_{|K|}\}$ with $q_i \in D_i$. Take a representative $K_0$ of $K$ in $\E^n$ with $q_1 \in P_0.$ First, we note that $K_0$ is contained in a ball of radius $R(P_0) + R_K$ centered at the origin, since the diameter of $K_0$ is at most $R_K$. We furthermore note that $R(P_0) \lesssim Ln2^{n/2} \lesssim R_Kn^42^{n/2}$, and by extension $R(P_0)+R_K \lesssim R_Kn^42^{n/2} \lesssim R_K2^n$ as well.

Next, we observe that each Voronoi cell touching this ball of radius $R_K + R(P_0)$ is contained in a ball of radius $R_K + R(P_0)+1  \le C_1R_K2^n$ for some absolute constant $C_1$. Thus, by a volume argument, there are $$\lesssim\frac{Vol(B^n_{C_1R_K2^n})}{Vol(V_0)} \lesssim\frac{n^2Vol(B^n_{C_1R_K2^n})}{Vol(B_{1/2}^n)} \lesssim R_K^n3^{n^2}$$ such cells. Each cell, furthermore, is bounded by a collection of hyperplanes, each of which corresponds to a point that is distance at most $1$ from the center of the Voronoi cell. Lemma \ref{lem:short_vec} tells us that $\Lambda$ is $\gtrsim \frac{1}{n^3}$-separated, so using lemma \ref{lem:s_t_sep}, we get that there are at most $\lesssim (C_2n^3)^n$ such hyperplanes for any given Voronoi cell, for some absolute constant $C_2 > 0$. Thus, the boundaries of the Voronoi cells in this ball of radius $R_K + R(P_0)+1$ are defined by a total of at most $\lesssim R_K^n3^{n^2}(C_2n^3)^n \lesssim R_K^n4^{n^2}$ hyperplanes. Now, we follow the same argument as in lemma \ref{lem:lowd_ad}; that is, we fix an affine basis of at most $n+1$ elements of $K$, and use this to map $D_1,\dots,D_{|K|}$ to a sign pattern of at most $\lesssim |K|R_K^n4^{n^2}$ hyperplanes in at most $n(n+1) \lesssim 2n^2$ variables. As in lemma \ref{lem:lowd_ad}, this map is injective. Thus, the number of such sign patterns provides a bound on the number of admissible collections of cells. Using theorem \ref{thm:m_t}, we get that there is an absolute constant $C_3 > 0$, such that the number of these sign patterns is at most $$\left( \frac{50(C_3R_K^n4^{n^2})|K|}{2n^2} \right)^{2n^2}  \lesssim R_K^{2n^3}4^{3n^4}|K|^{2n^2}.$$ This completes the proof.
\end{proof}

For the remainder of the proof of theorem \ref{thm:main_bign}, we let $Z$ denote $|z(D)|$ for some $D \in \mathcal{D}_\Lambda$ (as before, our choice of $D$ does not matter), and set $p = \min\{1/(2Z),1/(4C_K) \}$.

\begin{lemma}\label{lem:high_d_prob}
    The probability that an admissible collection of cells is all colored blue is at most $e^{-|K|p/4}$
\end{lemma}

\begin{proof}
    We will follow the same outline as in the proof of lemma \ref{lem:R2_allblue}. That is, we let $D_1,\dots,D_{|K|}$ be an admissible collection of cells, with points $q_1,\dots,q_{|K|}$ forming a copy of $K$ with $q_i \in D_i$. Then, we'll let $X_i$ be the indicator variable for the event that $D_i$ is colored red, and set $X = \sum_i X_i$. We say that $i \sim j$ if $||q_i - q_j|| < 5$. As in lemma \ref{lem:R2_allblue}, we can use lemma \ref{lem:5sep} to show this is a valid dependency graph.
    
    To apply theorem \ref{thm:Suen}, we compute $\mu = |K|p(1-p)^{Z}$ and $\Delta \le |K|C_K p^2(1-p)^{Z}$ and $\delta \le C_K p(1-p)^{Z} \le C_Kp$. This gives 

    \begin{align*}
        \mathbb{P}(X=0) &\le e^{-\mu + \Delta e^{2\delta}} \\
        &\le e^{-|K|p(1-p)^{Z} + |K|C_Kp^2(1-p)^{Z} e^{2C_{K}p}}\\
        &=e^{-|K|p(1-p)^{Z} (1-C_Kpe^{2C_Kp})}.
    \end{align*}
    Since $C_Kpe^{2C_Kp} \le e^{1/2}/4 \le 1/2 $ and $(1-p)^Z \ge 1/2$, we conclude 

    \begin{align*}
        \mathbb{P}(X=0) \le e^{-\mu + \Delta e^{2\delta}} &\le e^{-|K|p/4}.
    \end{align*}
\end{proof}
    Now, we need to put the bounds together. By lemmas \ref{lem:highd_ad} and \ref{lem:high_d_prob}, the expected number of admissible collections of cells that are colored all blue is 

\begin{align*}
    &\lesssim e^{-|K|p/4}R_K^{2n^3}4^{3n^4}|K|^{2n^2} \\
    &= e^{
        -\frac{|K|p}{4}
        +2n^3\log R_K
        +3n^4\log 4
        +2n^2\log |K|
    } \\
    &\le e^{
        -\frac{|K|p}{4}
        +2n^3\log R_K
        +10n^4
        +2n^2\log |K|
    }
\end{align*}
and thus, as long as $-|K|p/4 +2n^3\log R_K + 10n^4 + 2n^2\log |K| < -C$ for some absolute constant $C>0$, this expected number is less than one. Taking, for example, $|K| \gtrsim n^6\log R_K \max\{5^n,C_K\}$ will suffice. Then, if the expected number of admissible collections colored all blue is less than $1$, there must exist a coloring where no copy of $K$ is colored all blue. This completes the proof of Theorem \ref{thm:main_bign}. 

\medskip

\noindent \textbf{Acknowledgments:} We would like to thank Yuveshen Mooroogen for organizing the UBC SRP, where this research got started. We would also like to thank Arya Amin for working with us in the initial stage. Finally, we would like to acknowledge the use of AI in the course of our research. Various models of ChatGPT were used for brainstorming, literature search, and editing. In particular, the idea for the proof in section \ref{sec:packing} was generated in collaboration with AI. The manuscript was written in its entirety by the authors, who take full responsibility for its contents.

\bibliographystyle{abbrv}
\bibliography{references}

\end{document}